\ifx \labelsONmargin\undefined

  \ifx\RequirePackage\undefined
    \expandafter\ifx\csname amsart.sty\endcsname\relax
        \documentstyle[12pt,amssymb,amscd]{amsart}
    \fi

    \def\atSign{@@}

    \def\mathbb{\Bbb}
    \def\mathfrak{\frak}
    \def\mathbf{\bold}
    \ifx\boldsymbol\undefined	
      \def\boldsymbol#1{{\bold #1}}
    \fi

    \def\mathbit{\boldsymbol}

    \makeatletter
    \newenvironment{proof}{%
         \@ifnextchar[{%
                       \expandafter\let\expandafter\end@proof
                         \csname endpf*\endcsname
                         \my@proof
                      }{\let\end@proof\endpf\pf}%
        }{\end@proof}
    \def\my@proof[#1]{\@nameuse{pf*}{#1}}
    \makeatother

    \def\xrightarrow[#1]#2{@>{#2}>{#1}>}
    \def\xleftarrow[#1]#2{@<{#2}<{#1}<}

    \def\providecommand#1{\def#1}
    \def\emph#1{{\em #1}}
    \def\textbf#1{{\bf #1}}

    \def\mathring{\overset{\,\,{}_\circ}}
  \else
      \ifx\usepackage\RequirePackage
      \else
        \documentclass[12pt]{amsart}
        \usepackage{amssymb,amscd}
      \fi
      \ifx \mathring\undefined		
        \DeclareMathAccent{\mathring}{\mathalpha}{operators}{"17}
      \fi

      \long\def\FAKEendPROOF{\endtrivlist}
      \ifx\endproof\FAKEendPROOF
	  \def\endproof{\qed\endtrivlist}
      \fi

      \ifx\mathbit\undefined	
        \DeclareMathAlphabet{\mathbit}{OML}{cmm}{b}{it}
      \fi

      \def\atSign{@}

      \advance\oddsidemargin -0.1\textwidth
      \evensidemargin\oddsidemargin
      \def\Sb#1\endSb{_{\substack{#1}}}
      \def\Sp#1\endSp{^{\substack{#1}}}

  \fi
\makeatletter
\@ifundefined{DeclareFontShape} 
     {\@ifundefined{selectfont}
             {
                \message{We need AmS-LaTeX, this version of LaTeX does not 
                        support it}\@@end
             }{
                \def\mathcal{\cal}
                
                \def\pcyr{%
                        \def\default@family{UWCyr}%
                        \let\oldSl@\sl
                        \def\sl{\def\default@shape{it}\oldSl@}%
                        \cyracc
                        \language\Russian\family{UWCyr}\selectfont
                }
             }}{
                \DeclareFontEncoding{OT2}{}{\noaccents@}
                \DeclareFontSubstitution{OT2}{cmr}{m}{n}
                \DeclareFontFamily{OT2}{cmr}{\hyphenchar\font45 }
                \DeclareFontShape{OT2}{cmr}{m}{n}{%
                     <5><6><7><8><9><10>gen*wncyr %
                     <10.95><12><14.4><17.28><20.74><24.88> wncyr10 %
                }{}
                \DeclareFontShape{OT2}{cmr}{m}{it}{%
                     <5><6><7><8><9><10> gen * wncyi%
                     <10.95><12><14.4><17.28><20.74><24.88> wncyi10%
                }{}
                \DeclareFontShape{OT2}{cmr}{bx}{n}{%
                     <5><6><7><8><9><10> gen * wncyb%
                     <10.95><12><14.4><17.28><20.74><24.88> wncyb10%
                }{}
                \DeclareFontShape{OT2}{cmr}{m}{sl}{%
                     <-> ssub * cmr/m/it%
                }{}
                \DeclareFontShape{OT2}{cmr}{m}{sc}{%
                     <5><6><7><8><9><10>%
                     <10.95><12><14.4><17.28><20.74><24.88> wncysc10%
                }{}
                \DeclareFontFamily{OT2}{cmss}{\hyphenchar\font45 }
                \DeclareFontShape{OT2}{cmss}{m}{n}{%
                     <8><9><10> gen * wncyss%
                     <10.95><12><14.4><17.28><20.74><24.88> wncyss10%
                }{}
                
                \def\cyrencodingdefault{OT2}
                \def\pcyr{%
                        \cyracc
                        \let\encodingdefault\cyrencodingdefault
                        \language\Russian\fontencoding{OT2}\selectfont
                }
             }   

\@ifundefined{theorembodyfont} 
     {
        \def\theorembodyfont#1{\relax}
        \makeatletter
          \let\@@th@plain\th@plain
          \def\th@plain{ \@@th@plain \slshape }
        \makeatother
        \let\normalshape\relax
     }{}   

\ifx\cprime\undefined
     \def\cprime{$'$}
\fi
\makeatletter
  \def\@sect@my#1#2#3#4#5#6[#7]#8{%
\ifnum #2>\c@secnumdepth
   \let\@svsec\@empty
 \else
   \refstepcounter{#1}%
\edef\@svsec{\ifnum#2<\@m
             \@ifundefined{#1name}{}{\csname #1name\endcsname\ }\fi
\noexpand\rom{\csname the#1\endcsname.}\enspace}\fi
 \@tempskipa #5\relax
 \ifdim \@tempskipa>\z@ 
   \begingroup #6\relax
   \@hangfrom{\hskip #3\relax\@svsec}{\interlinepenalty\@M #8\par}%
   \endgroup
   \if@article\else\csname #1mark\endcsname{%
        \ifnum \c@secnumdepth >#2\relax\csname the#1\endcsname. \fi#7}\fi
\ifnum#2>\@m \else
       \let\@tempf\\ \def\\{\protect\\}\addcontentsline{toc}{#1}%
{\ifnum #2>\c@secnumdepth \else
             \protect\numberline{%
               \ifnum#2<\@m
               \@ifundefined{#1name}{}{\csname #1name\endcsname\ }\fi
               \csname the#1\endcsname.}\fi
           #8}\let\\\@tempf
     \fi
 \else
  \def\@svsechd{#6\hskip #3\@svsec
    \@ifnotempty{#8}{\ignorespaces#8\unskip
       \ifnum\spacefactor<1001.\fi}%
        \ifnum#2>\@m \else
          \let\@tempf\\ \def\\{\protect\\}\addcontentsline{toc}{#1}%
            {\ifnum #2>\c@secnumdepth \else
              \protect\numberline{%
                \ifnum#2<\@m
                \@ifundefined{#1name}{}{\csname #1name\endcsname\ }\fi
                \csname the#1\endcsname.}\fi
             #8}\let\\\@tempf\fi}%
 \fi
\@xsect{#5}}
  \ifx\RequirePackage\undefined
  \let\@sect\@sect@my             
  \fi
  \ifx\RequirePackage\undefined	
  \def\th@remark@my{\theorempreskipamount6\p@\@plus6\p@
    \theorempostskipamount\theorempreskipamount
    \def\theorem@headerfont{\it}\normalshape}
    \let\th@remark\th@remark@my
  \else				
    \let\o@@remark\th@remark

    \ifx\theorempostskipamount\undefined		
    \else						
      \def\th@remark{\o@@remark
	\ifdim\theorempostskipamount < 2pt\relax
	  \theorempostskipamount\theorempreskipamount
	     \multiply\theorempostskipamount\tw@
	     \divide\theorempostskipamount\thr@@
	\fi
      }
    \fi
  \fi
\let\myLabel\@gobble
\def\labelsONmargin{\@mparswitchfalse\def\myLabel##1{\@bsphack\marginpar
                                  {\normalshape\tiny\rm Label ##1}\@esphack}}
\ifx \url\undefined
  \def\url#1{{\tt #1}}%
\fi
\def\cyracc{\def\u##1{
                \if \i##1\char"1A%
                \else \if I##1\char"12%
                \else \accent"24 ##1\fi\fi }%
\def\"##1{\if e##1{\char"1B}%
                \else \if E##1{\char"13}%
                \else \accent"7F ##1\fi\fi }%
\def\9##1{\if##1z\char"19 
\else\if##1Z\char"11 
\else\if##1E\char"03 
\else\if##1e\char"0B 
\else\if##1u\char"18 
\else\if##1U\char"10 
\else\if##1A\char"17 
\else\if##1a\char"1F 
\else\if##1p\char"7E 
\else\if##1P\char"5E 
\else\if##1Q\char"5F 
\else\if##1q\char"7F 
\else\if##1i\char"1A 
\else\if##1I\char"12 
\else\if##1N\char"7D 
\fi
\fi
\fi
\fi
\fi
\fi
\fi
\fi
\fi
\fi
\fi
\fi
\fi
\fi
\fi
}%
\def\cydot{{\kern0pt}}}%
\def\cydot{$\cdot$}

\ifx\Russian\undefined
        \def\Russian{0\relax
    \message{Don't know the hyphenation rules for Russian^^J
                        Please do INITeX with `input  russhyph' in the 
                        command line}%
                \gdef\Russian{0\relax}%
        }
\fi
\ifx\RequirePackage\undefined			
  \def\@putname#1#2#3#4{\def\@@ref{#3}\let\old@bf\bf
        \def\bf##1{\old@bf\if?\noexpand##1?{#4}\else##1\fi}%
	#1{#2}%
        \let\bf\old@bf}
\else						
  \def\@putname#1#2#3#4{\def\@@ref{#3}\let\old@bf\bf	
	\let\old@reset@font\reset@font			
        \def\bf##1{\old@bf\if?\noexpand##1?{#4}\else##1\fi}%
	\def\reset@font##1##2{\old@reset@font##1\if?\noexpand##2?{#4}\else##2\fi}#1{#2}%
        \let\bf\old@bf\let\reset@font\old@reset@font}
\fi
\let\my@ref=\ref
\def\ref#1{\@putname\my@ref{#1}{#1}{\tiny\rm\@@ref}}
\let\my@pageref=\pageref
\def\pageref#1{\@putname\my@pageref{#1}{#1}{\tiny\rm\@@ref}}
\let\my@cite=\cite
\def\cite#1{\@putname\my@cite{#1}{\@citeb}{\tiny\rm\@@ref}}
\makeatother
\fi
\relax 
\theoremstyle{plain} 
\theorembodyfont{\sl}

\numberwithin{equation}{section}
\theorembodyfont{\rm}
\theoremstyle{definition}

\newtheorem{definition}{Definition}[section]

\newtheorem{example}[definition]{Example}

\theoremstyle{remark}

\newtheorem{remark}[definition]{Remark} 

\theoremstyle{plain} 
\theorembodyfont{\sl}

\newtheorem{theorem}[definition]{Theorem}
\newtheorem{lemma}[definition]{Lemma}
\newtheorem{corollary}[definition]{Corollary}
\newtheorem{proposition}[definition]{Proposition}

\renewcommand{\dim}{\mathrm{dim}}
\newcommand{\sdim}{\mathrm{sdim}}

\newcommand{\soc}{\mathrm{soc}}
\newcommand{\id}{\mathrm{id}}
\newcommand{\wt}{\mathrm{wt}}

\newcommand{\Hom}{\mathrm{Hom}}
\newcommand{\Ext}{\mathrm{Ext}}

\renewcommand{\Im}{\mathrm{Im}}
\newcommand{\Ker}{\mathrm{Ker}}

\newcommand{\Rep}{\mathrm{Rep}}

\begin{document}
\bibliographystyle{amsplain}
\relax 

\title[Tensor product of the Fock representation with its dual]{ Tensor product of the Fock representation with its dual and the Deligne category}

\author{ Vera Serganova }

\date{ \today }

\address{ Dept. of Mathematics, University of California at Berkeley,
Berkeley, CA 94720 }

\email{serganov\atSign{}math.berkeley.edu}

\dedicatory{To Kolya Reshetikhin for his 60th birthday}
\maketitle
\section{Introduction} Let $\mathbb V:=\mathbb C^{\mathbb Z}$ be a countable dimensional vector space with fixed basis $\{u_i\mid i\in \mathbb Z\}$.
Consider the Lie algebra $\mathfrak{sl}(\infty)$ of all traceless linear operators in $\mathbb C^{\mathbb Z}$ annihilating almost all $u_i$.
Clearly, $\mathfrak{sl}(\infty)$ can be identified with the Lie algebra of traceless infinite matrices with finitely many non-zero entries.
We consider $\mathfrak{sl}(\infty)$ as a Kac-Moody Lie algebra associated with Dynkin diagram $A_\infty$. The Chevalley--Serre generators $e_a,f_a,\ a\in\mathbb Z$ of $\mathfrak{sl}(\infty)$ act on $\mathbb V$ by
$$f_a u_b=\delta_{a,b}u_{b+1},\quad e_a u_b=\delta_{a+1,b}u_{b-1}.$$
The fermionic Fock space $\mathfrak F$ is a simple $\mathfrak{sl}(\infty)$-module with fundamental highest weight $\omega_{-1}$. It has a realization as the
``semi-infinite exterior power'' $\Lambda^{\infty/2}\mathbb C^{\mathbb Z}$ which is the span of all formal expressions $u_{i_1}\wedge u_{i_2}\wedge \dots $
satisfying the conditions $i_j>i_{j+1}$ for all $j\geq 1$ and $i_k=-k$ for sufficiently large $k$. In this way the highest weight vector is
$u_{\emptyset}:=u_{-1}\wedge u_{-2}\wedge\dots $.
The famous boson-fermion correspondence identifies $\mathfrak F$ with the space of symmetric functions. That in particular implies that $\mathfrak F$ has a natural
basis  $\{u_\lambda\}$ enumerated by partitions $\lambda$ (this basis corresponds to Schur functions) where
$$u_\lambda:=u_{\lambda_1-1}\wedge u_{\lambda_2-2}\wedge u_{\lambda_3-3}\wedge\dots.$$

Let $t\in\mathbb Z$. We denote by $\mathfrak F^\vee_t$ the simple  $\mathfrak{sl}(\infty)$-module with lowest weight $-\omega_{t-1}$. We will use the
following realization
of  $\mathfrak F^\vee_t$. Set $\mathbb V^\vee=\mathbb C^{\mathbb Z}$ with basis $\{w_i\mid i\in\mathbb Z\}$ and define the action of $e_a,f_a$ on $\mathbb V^\vee$ by
$$e_a w_b=\delta_{a,b}w_{b+1},\quad f_a w_b=\delta_{a+1,b}w_{b-1}.$$ Then $\mathfrak F^\vee_t$ is the span of all formal expressions $w_{i_1}\wedge w_{i_2}\wedge\dots $
satisfying the conditions $i_j>i_{j+1}$ for all $j\geq 1$ and $i_k=t-k$ for sufficiently large $k$. We can enumerate the elements of the basis of $\mathfrak F_t^\vee$ by partitions 
$$w_\mu:=w_{\mu_1+t-1}\wedge w_{\mu_2+t-2}\wedge w_{\mu_3+t-3}\wedge\dots.$$

The goal of this paper is to describe the structure of $\mathfrak F^\vee_t\otimes\mathfrak F$. Let us consider $(m,n)\in\mathbb Z^2$ such that $m-n=t$.
As follows from \cite{PS} $\Lambda^{m}\mathbb V^\vee\otimes \Lambda^n\mathbb V$ is an indecomposable $\mathfrak{sl}(\infty)$-module with simple socle
$S_{m,n}$. To describe this socle consider the contraction map $c:\mathbb V\otimes\mathbb V^\vee\to\mathbb C$ given by $c(w_i\otimes u_j)=(-1)^j\delta_{i,j}$
and extend it to $c_{m,n}:\Lambda^{m}\mathbb V^\vee\otimes \Lambda^n\mathbb V\to \Lambda^{m-1}\mathbb V^\vee\otimes \Lambda^{n-1}\mathbb V$. Then $S_{m,n}$ is the kernel
of $c_{m,n}$.

\begin{theorem}\label{main}
  \begin{enumerate}
\item  The $\mathfrak{sl}(\infty)$-module  $\mathfrak R:=\mathfrak F^\vee_t\otimes\mathfrak F$ has an infinite decreasing
  filtration
  $$\mathfrak R:=\mathfrak R^0\supset \mathfrak R^1\supset\dots \supset\mathfrak R^k\supset \dots$$
  such that $\cap_k \mathfrak R^k=0$ and
  $$\mathfrak R^k/\mathfrak R^{k+1}
    \simeq \begin{cases}S_{k+t,k}\,\,\text{if}\,\,t\geq 0,\\S_{k,k-t}\,\,\text{if}\,\,t<0.\end{cases}$$
  \item Every non-zero submodule of $\mathfrak R$ coincides with $\mathfrak R^r$ for some $r\geq 0$.
    \end{enumerate}
\end{theorem}

The proof of this theorem is based on categorification of $\mathfrak F^\vee_t\otimes\mathfrak F$ by the complexified Grothendieck group $K[\mathcal V_t]_{\mathbb C}$
of the abelian
envelope $\mathcal V_t$ of the Deligne category $\operatorname{Rep}GL_t$ explained in \cite{E} and Brundan categorification of
$\Lambda^{m}\mathbb V^\vee\otimes \Lambda^n\mathbb V$ via representation theory of the supergroup $GL(m|n)$, \cite{B}. We use the symmetric monoidal functor
$$DS_{m,n}: \mathcal V_t\to \operatorname{Rep}GL(m|n)$$
for $m-n=t$.  Existence of such functor follows from construction of $\mathcal V_t$, see \cite{EHS}.
While $DS_{m,n}$ is not exact, it has a certain  property, see Lemma \ref{Hinich} below, which allows to define the linear map
$$ds_{m,n}:K[\mathcal V_t]_{\mathbb C}\to K_{red}[\operatorname{Rep}GL(m|n)]_{\mathbb C}$$ where by $K_{red}$ we denote the quotient of the Grothendieck
group $K$ by the relation
$[\mathbb C^{0|1}]=-[\mathbb C]$ in the category $\operatorname{Rep}GL(m|n)$. Furthermore, $ds_{m,n}$ is a homomorphism of rings and also a homomorphism
of $\mathfrak{sl}(\infty)$-modules. We prove that  the quotients $\Ker ds_{m-1,n-1}/\Ker ds_{m,n}$ form the layers of the radical filtration of
$\mathfrak F^\vee_t\otimes\mathfrak F\simeq K[\mathcal V_t]_{\mathbb C}$. Let us warn the reader that the image of $ds_{m,n}$ is not
$\Lambda^{m}\mathbb V^\vee\otimes \Lambda^n\mathbb V$ but another submodule in  $K_{red}[\operatorname{Rep}GL(m|n)]_{\mathbb C}$. While this
submodule has the same Jordan-Hoelder series
as $\Lambda^{m}\mathbb V^\vee\otimes \Lambda^n\mathbb V$, it is not isomorphic to
$\Lambda^{m}\mathbb V^\vee\otimes \Lambda^n\mathbb V$ as an $\mathfrak{sl}(\infty)$-module.

The second part of the paper contains calculation of dimensions of certain objects in $\mathcal V_t$.

The author was supported by NSF grant DMS-1701532. The author would like to thank Inna Entova-Aizenbud for reading the first version of the paper and pointing out
typos and unclear arguments.

\section{The category  $\operatorname{Rep}GL(m|n)$ and $DS$ functors}

\subsection{Translation functors} Let $\operatorname{Rep}GL(m|n)$ denote the category of finite-dimensional $GL(m|n)$-modules.
Let $\mu=(a_1,\dots,a_m|b_1,\dots b_n)\in\mathbb Z^{m+n}$
satisfy the condition $a_1\geq a_2\geq \dots\geq a_m, b_1\geq b_2\geq\dots\geq b_n$. For every such $\mu$ there are three canonical objects in  
$\operatorname{Rep}GL(m|n)$:
\begin{enumerate}
\item The simple module $S(\mu)$ with highest weight $\mu$;
\item The Kac module $K(\mu):=U(\mathfrak{gl}(m|n))\otimes_{U(\mathfrak{p})} S_0(\mu)$, where $\mathfrak p$ is the parabolic subalgebra with Levi subalgebra
    $\mathfrak{gl}(m|n)_{\bar 0}$, $S_0(\mu)$ is the  simple $\mathfrak{gl}(m|n)_{\bar 0}$-module with highest weight $\mu$;
    \item The indecomposable projective cover $P(\mu)$ of $S(\mu)$. 
    \end{enumerate}

    The category   $\operatorname{Rep}GL(m|n)$ is the highest weight category, \cite{Z}. We denote by $K_{red}[\operatorname{Rep}GL(m|n)]$ the reduced
  Grothendieck group of   $\operatorname{Rep}GL(m|n)$ and set $$J_{m|n}:=K_{red}[\operatorname{Rep}GL(m|n)]\otimes_{\mathbb Z}\mathbb C.$$

It was a remarkable discovery of J. Brundan that $J_{m|n}$ has a natural structure of $\mathfrak{sl}(\infty)$-module, \cite{B}. To define it let us consider
translation functors
$E_a, F_a : \operatorname{Rep}GL(m|n)\to \operatorname{Rep}GL(m|n)$ defined in the following way. There is a canonical $\mathfrak{gl}(m|n)$-invariant map
$\omega:\mathbb C\to\mathfrak{gl}(m|n)\otimes \mathfrak{gl}(m|n)$ usually called the Casimir element. Let $V_{m|n}$ be the standard $GL(m|n)$-module and
$M$ be an arbitrary
object of $\operatorname{Rep}GL(m|n)$. Let $\Omega$ be the composition map
$$\mathbb C\otimes M\otimes V_{m|n}\xrightarrow{\omega\otimes\id}\mathfrak{gl}(m|n)\otimes \mathfrak{gl}(m|n) \otimes M\otimes V_{m|n}
\xrightarrow{\id\otimes s\otimes\id}$$
$$ \mathfrak{gl}(m|n)\otimes M\otimes\mathfrak{gl}(m|n)\otimes V_{m|n}\xrightarrow{a_M\otimes a_{V_{m|n}}} M\otimes V_{m|n},
$$
where $s$ is the braiding in  $\operatorname{Rep}GL(m|n)$ defined by the sign rule and $a_M,a_{V_{m|n}}$ are the action maps.
Let $E_a(M)$ be the generalized eigenspace  of $\Omega$ in $M\otimes V_{m|n}$ with eigenvalue $a$.
Similarly, we define $F_a(M)$  as the generalized eigenspace  of $\Omega'$ in $M\otimes V^*_{m|n}$ with eigenvalue $a$, where $\Omega'$ is defined as
above with substitution of $V_{m|n}^*$ in place of $V_{m|n}$.

The following theorem is a direct consequence of results in \cite{B}.

\begin{theorem}\label{B}
  \begin{enumerate}
  \item $E_a$, $F_a$ are non-zero only for $a\in\mathbb Z$;
  \item $E_a$, $F_a$ are biadjoint exact endofunctors of $\operatorname{Rep}GL(m|n)$;
  \item Let $e_a,f_a: J_{m|n}\to J_{m|n}$  be the induced $\mathbb C$-linear maps. Then $e_a,f_a$ satisfy the Chevalley-Serre relations for $A_{\infty}$.
    Hence $J_{m|n}$ is an $\mathfrak{sl}(\infty)$-module.
  \item The subspace of $\Lambda_{m|n}\subset J_{m|n}$ generated by classes of all Kac modules $[K(\mu)]$ is an $\mathfrak{sl}(\infty)$-submodule isomorphic to
    $\Lambda^{m}\mathbb V^\vee\otimes \Lambda^n\mathbb V$.
  \end{enumerate} 
\end{theorem}

We need the exact description of the socle filtration of $J_{m|n}$ obtained in \cite{HPS}, Corollary 29.

\begin{proposition}\label{HPS} The $\mathfrak{sl}(\infty)$-module $J_{m|n}$ has finite length. Furthermore, the socle filtration of $J_{m|n}$ is given by the formula  
  $$\soc^i(J_{m|n})/\soc^{i-1}(J_{m|n})\simeq S_{m-i+1|n-i+1}^{\oplus i}.$$
  In particular, the socle of $J_{m|n}$ is a simple $\mathfrak{sl}(\infty)$-module isomorphic to $S_{m|n}$. It is identified with the subspace generated by classes of
  all projective modules $[P(\mu)]$.
  \end{proposition}

  \subsection{$DS$-functor} Fix  an odd $x\in\mathfrak{gl}(m|n)_{\bar 1}$ such that $[x,x]=0$ and $\operatorname{rk} x=1$. Define a functor
  $DS_x$ from $\operatorname{Rep}GL(m|n)$ to the category of vector superspaces by setting
  $$DS_x(M)=\Ker x_M/\Im x_M.$$
  It is shown in \cite{DS} that  $M_x$ has a natural structure of $GL(m-1|n-1)$-module and $DS_x$ is a symmetric monoidal functor
  $$\operatorname{Rep}GL(m|n)\to\operatorname{Rep}GL(m-1|n-1).$$
  Furthermore, although $DS_x$ is not an exact functor it has the following property pointed out by V. Hinich. For the proof see \cite{HPS} Lemma 30.
  \begin{lemma}\label{Hinich} Every exact sequence
    $0\to N\to M\to K\to 0$
 of $GL(m|n)$-modules    induces the exact sequence
    $$0\to E\to DS_x N\to DS_x M\to DS_xK\to E'\to 0$$
    for certain $E\in \operatorname{Rep}GL(m-1|n-1)$ and $E'\simeq E\otimes\mathbb C^{0|1}$.
    \end{lemma}
    It follows immediately from Lemma \ref{Hinich} that $DS_x$ induces a homomorphism of complexified reduced Grothendieck groups $ds_x:J_{m|n}\to J_{m-1|n-1}$.
    While $DS_x$ and $DS_y$ are
    not isomorphic if $x$ and $y$ are not conjugate by the adjoint action of $GL(m)\times GL(n)$, the homomorphism $ds_x$ does not depend on a choice of $x$.
    In \cite{HR} the homomorphism $ds_x$ was constructed explicitly in terms of supercharacters and the kernel of $ds_x$ was computed.
    \begin{lemma}\label{HR}
      \begin{enumerate}
      \item $DS_x$ commutes with translation functors $E_a,F_a$ and hence $DS_x$ induces  a homomorphism
        $ds_x: J_{m|n}\to J_{m-1|n-1}$ of $\mathfrak{sl}(\infty)$-modules.
        \item The kernel of $ds_x$ coincides with $\Lambda_{m|n}$.  
        \end{enumerate}
      \end{lemma}
      \begin{proof} For (1) see Lemma 32 in \cite{HPS}. For (2) see \cite{HR}.
        \end{proof}

        \section{The category $\mathcal V_t$, translation functors and categorification}
        \subsection {The Deligne category $\mathcal D_t$} In \cite{DM} Deligne and Milne constructed a family  $\{D_t=\operatorname{Rep} GL_t\mid t\in\mathbb C\}$
        of symmetric monoidal rigid categories satisfying the following properties:
        \begin{enumerate}
        \item $\mathcal D_t$ is a universal additive symmetric monoidal Karoubian category generated by a dualizable object $V_t$ of dimension $t$;
        \item The indecomposable objects of $\mathcal D_t$ are in bijection with bipartitions $\lambda=(\lambda^\bullet,\lambda^\circ)$, we denote the
          corresponding indecomposable objects by $T(\lambda)$;
        \item If $t\notin\mathbb Z$, then $\dim\Hom (T(\lambda), T(\nu))=\delta_{\lambda,\mu}$ and hence the category $\mathcal D_t$ is an abelian semisimple category;
        \item If $t\in\mathbb Z$, and $m-n=t$, then there exists a (unique up to isomorphism) symmetric monoidal functor
          $F_{m|n}: \mathcal D_t\to\operatorname{Rep}GL(m|n)$ which sends $V_t$ to $V_{m|n}$. This functor is full. 
          \end{enumerate}
          The functor $F_{m|n}$ was studied in \cite{CW}. In particular, it was computed on the indecomposable objects of $\mathcal D_t$.
          We call a bipartition $\lambda=(\lambda^\bullet,\lambda^\circ)$
          an $(m|n)$-cross if for there exists $0\leq k\leq m$ such that $\lambda^\bullet_{k+1}+(\lambda^\circ)^T_{m-k+1}\leq n$.
          Here $\mu^T$ stands for the conjugate of $\mu$.
          Denote by $C(m|n)$ the set of all $(m|n)$-crosses.
          \begin{theorem}\label{CW}
            \begin{enumerate}
              \item $F_{m|n}T(\lambda)\neq 0$ if and only if $\lambda\in C(m|n)$.
              \item The set $\{F_{m|n}T(\lambda)\mid \lambda\in C(m|n)\}$ is a complete set of pairwise non-isomorphic indecomposable direct summands in
                tensor powers $V_{m|n}^{\otimes p}\otimes (V_{m|n}^*)^{\otimes q}$ for $p,q\geq 0$.
              \end{enumerate}
            \end{theorem}
            \begin{proof} The first statement is Theorem 8.7.6 in \cite{CW} and the second is the particular case of Theorem 4.7.1 in \cite{CW}.
            \end{proof}
            \subsection{The abelian envelope of $\mathcal D_t$} Let $t\in\mathbb Z$. Then $\mathcal D_t$ is not abelian. In \cite{EHS} we construct an abelian envelope
            $\mathcal V_t$ of $\mathcal D_t$. We need here some particular features of this construction. Let $m-n=t$ and let
            $\operatorname{Rep}^kGL(m|n)$ be the abelian full subcategory of $\operatorname{Rep}GL(m|n)$ containing mixed tensor powers
            $V_{m|n}^{\otimes p}\otimes (V_{m|n}^*)^{\otimes q}$ for $p,q\leq k$. The following statement is crucial for our construction.
            \begin{lemma}\label{EHS} Let $m,n>> k$ and $x\in \mathfrak{gl}(m|n)_{\bar 1}$ be a self-commuting element of rank $1$. Then the restriction of
              $DS_x$ to $\operatorname{Rep}^kGL(m|n)$ defines an equivalence of the categories
              $\operatorname{Rep}^kGL(m|n)\to\operatorname{Rep}^kGL(m-1|n-1)$.
            \end{lemma}
            That allows us to define the abelian category $\mathcal V^k_t$ as the inverse limit $\displaystyle\lim_{\leftarrow}\operatorname{Rep}^kGL(m|n)$.
            Then set
            $$\mathcal V_t:=\lim_{\rightarrow}\mathcal V^k_t.$$
            We have an exact fully faithful functor $I:\mathcal D_t \to\mathcal V_t$. Slightly abusing notation we write $T(\lambda)=IT(\lambda)$.
            \begin{lemma}\label{functor} For every $(m|n)$ such that $m-n=t$ there exists a symmetric monoidal functor
              $DS_{m|n}:\mathcal V_t\to \operatorname{Rep}GL(m|n)$. This functor is not exact but satisfies the condition of Lemma \ref{Hinich}.
              Moreover, $DS_{m|n}\circ I$ is isomorphic to $F_{m|n}$.
            \end{lemma}
            \begin{proof} It suffices to construct $DS_{m|n}:\mathcal V_t^k\to  \operatorname{Rep}GL(m|n)$. We identify $V_t^k$ with  $\operatorname{Rep}^kGL(m'|n')$
              for sufficiently large $m',n'$ and define $DS_{m|n}:\operatorname{Rep}^kGL(m'|n')\to \operatorname{Rep}^kGL(m|n)$ as a composition
              of the functors $DS_{x_r}\circ DS_{x_{r-1}}\circ\dots DS_{x_1}$ for some self-commuting rank $1$ odd elements $x_i\in \mathfrak{gl}(m+i|n+i)$
                with $r=m'-m=n'-n$. Lemma \ref{EHS} ensures that this composition does not depend on
              the choice of $(m'|n')$ and that passing to the direct limit is well-defined. By construction $DS_{m|n}$ satisfies Lemma \ref{Hinich}.
              Finally, $DS_{m|n}\circ I$ is a symmetric monoidal functor from $\mathcal D_t$ to $\operatorname{Rep}^kGL(m|n)$ which maps $V_t$ to $V_{m|n}$.
              Hence by (4) it must be isomorphic to $F_{m|n}$.
              \end{proof}
              \begin{remark}\label{unique?} Construction of  $DS_{m|n}$ given in the above proof depends on a choice of $x_s\in\mathfrak{gl}(m+s,n+s)_{\bar 1}$.
                Apriori there may be several non-isomorphic functors satisfying the condition of Lemma \ref{functor}. We suspect however that all these functors
                are isomorphic. Anyway as follows from the proof we can choose the sequence $DS_{m|n}$ so that
                $DS_{m-1|n-1}=DS_x\circ DS_{m|n}$ for some $x\in\mathfrak{gl}(m|n)_{\bar 1}$. Note that $DS_{m|n}T(\lambda)\simeq F_{m|n}T(\lambda)$, hence
                on tilting objects the image of $DS_{m|n}$ does not depend on the choice of $x_s$. Furthermore,
                $DS_{m|n}$ defines a homomorphism $ds_{m|n}: K[\mathcal V_t]_{\mathbb C}\to J_{m|n}$ which does not depend on a choice of $x_s$.
                \end{remark}

                \subsection{Objects of $\mathcal V_t$} There are three types of objects in $\mathcal V_t$ enumerated by bipartitions:
                \begin{itemize}
                \item Simple objects $L(\lambda)$, after identification with of $\mathcal V_t^k$ with $\Rep^kGL(m|n)$ the highest weight of the
                  corresponding representation is $\sum\lambda_i^\bullet\varepsilon_i-\sum (\lambda_i^\circ)\delta_i$ for the following set of simple roots
                  $\mathfrak{gl}(m|n)$: $\varepsilon_1-\varepsilon_2,\dots,\varepsilon_m-\delta_n,\delta_n-\delta_{n-1},\dots.\delta_2-\delta_1$.
                \item Standard objects $V(\lambda)$, those are maximal quotients of the  Kac modules lying in $\Rep^kGL(m|n)$. They can be described as images
                    of the irreducible module in $\Rep \mathfrak{gl}(\infty)$, see \cite{EHS}.
                    \item Indecomposable tilting objects $T(\lambda)$.
                    \end{itemize}

                    It is proven in \cite{EHS} that for every $k\geq 0$ the abelian category $\mathcal V^k_t$ is a highest weight category. Moreover, simple
                    standard and tilting objects do not depend on $k$ as soon as $k$ is sufficiently large.
                    In particular, $T(\lambda)$ has a filtration by $V(\mu)$ with the property:
\begin{equation}\label{mult}
  (T(\lambda):V(\lambda))=1,\quad(T(\lambda):V(\mu))\neq 0 \Rightarrow \lambda=\mu\,\,\text{or}\,\,\mu\subset\lambda.
\end{equation}
Here we say $\mu\subset\lambda$ if $\mu^\bullet$ is contained in $\lambda^\bullet$ and $\mu^\circ$ is contained in $\lambda^\circ$.
Furthermore, there is an interesting reciprocity, \cite{E}:
\begin{equation}\label{reciprocity}
  (T(\lambda):V(\mu))=[V(\lambda):L(\mu)].
  \end{equation}
                    It is shown in \cite{EHS} that $[V(\lambda):L(\mu)]\leq 1$. In \cite{E} all pairs $(\lambda,\mu)$ for which $[V(\lambda):L(\mu)]=1$
                    are described in terms of weight diagrams.
                    \begin{lemma}\label{grothgroup} All three sets
                      $\{[L(\lambda]\}$, $\{[V(\lambda)]\}$ and $\{[T(\lambda)]\}$ are bases  in the Grothendieck group $K[\mathcal V_t]$.
                      Furthermore there exists $K(\lambda,\mu)=0,1$ such that
                      $$[V(\lambda)]=\sum_{\mu\subset\lambda}K(\lambda,\mu)[L(\mu)],\quad [T(\lambda)]=\sum_{\mu\subset\lambda}K(\lambda,\mu)[V(\mu)].$$
                    \end{lemma}
                    \begin{proof} The second assertion is a consequence of (\ref{mult}) and (\ref{reciprocity}). The first assertion follows fom the fact that
                      $K(\lambda,\mu)$ is upper triangular matrix with respect to the partial order with rank function $|\lambda^\circ|+|\lambda^\bullet|$, and
                      $K(\lambda,\mu)$ has $1$-s on the main diagonal.
                      \end{proof}
                    \subsection{Translation functors and categorical action of $\mathfrak{sl}(\infty)$} One readily sees that $\mathfrak{gl}(V_t):=V_t\otimes V^*_t$
                    is a Lie algebra object in $\mathcal V_t$. Furthermore, there exists a unique canonical morphism
                    $\omega:\mathbf{1}\to \mathfrak{gl}(V_t)$. For every $X\in\mathcal V_t$ we do have the action morphism
                    $a_X:\mathfrak{gl}(V_t)\otimes X\to X$. Hence in the same way as for $\Rep GL(m|n)$ we can define the translation functors $E_aX$ and $F_aX$
                    as generalized eigenspaces with eigenvalue $a$ for
                    $$\Omega: X\otimes V_t\xrightarrow{\omega\otimes\id}\mathfrak{gl}(V_t)\otimes \mathfrak{gl}(V_t) \otimes X\otimes V_{t}
\xrightarrow{\id\otimes s\otimes\id}$$
$$ \mathfrak{gl}(V_t)\otimes X\otimes\mathfrak{gl}(V_t)\otimes V_{t}\xrightarrow{a_X\otimes a_{V_{t}}} X\otimes V_{t}
  $$
  and  $$\Omega': X\otimes V^*_t\xrightarrow{\omega\otimes\id}\mathfrak{gl}(V_t)\otimes \mathfrak{gl}(V_t) \otimes X\otimes V^*_{t}
\xrightarrow{\id\otimes s\otimes\id}$$
$$ \mathfrak{gl}(V_t)\otimes X\otimes\mathfrak{gl}(V_t)\otimes V_{t}^*\xrightarrow{a_X\otimes a_{V^*_{t}}} X\otimes V^*_{t},$$
respectively.

The following theorem is proven in \cite{E}

\begin{theorem}\label{E} Let $t\in\mathbb Z$.
  \begin{enumerate}
  \item $E_a$, $F_a$ are non-zero only for $a\in\mathbb Z$;
  \item $E_a$, $F_a$ are biadjoint exact endofunctors of $\mathcal V_t$;
  \item Let $e_a,f_a: K[\mathcal V_t]_{\mathbb C}\to K[\mathcal V_t]_{\mathbb C}$  be the induced $\mathbb C$-linear maps.
    Then $e_a,f_a$ satisfy the Chevalley-Serre relations for $A_{\infty}$.
    Hence $K[\mathcal V_t]_{\mathbb C}$ is an $\mathfrak{sl}(\infty)$-module.
  \item There is a unique isomorphism $f:K[\mathcal V_t]_{\mathbb C}\to \mathfrak F^\vee_t\otimes\mathfrak F$ of $\mathfrak{sl}(\infty)$-modules such that
    $f([V(\lambda)])=v_\lambda:=w_{\lambda^\bullet}\otimes  u_{\lambda^\circ}$.
  \end{enumerate} 
\end{theorem}
\section{Proof of the main theorem}
Recall the functor $DS_{m|n}$ defined in Lemma \ref{functor}.
\begin{lemma}\label{commuting} We have the following commutative diagrams of functors:
\begin{equation}
\begin{CD}
\mathcal V_t @>{E_a(F_a)}>>  \mathcal V_t\\
 @V{DS_{m|n}}VV @V{DS_{m|n}}VV \\
 \Rep GL(m|n)@>{E_a(F_a)}>> \Rep GL(m|n) 
\end{CD}\notag
\end{equation}
\end{lemma}
\begin{proof} By Lemma \ref{HR} one has the following commutative diagram
  \begin{equation}
\begin{CD}
\Rep GL(m|n) @>{E_a(F_a)}>>  \Rep GL(m|n)\\
 @V{DS_{x}}VV @V{DS_{x}}VV \\
 \Rep GL(m-1|n-1)@>{E_a(F_a)}>> \Rep GL(m-1|n-1) 
\end{CD}\notag
\end{equation}
Hence the statement follows from definition of $\mathcal V_t$ and the proof of Lemma \ref{functor}.
\end{proof}
\begin{corollary}\label{modhom} The induced map $ds_{m|n}: K[\mathcal V_t]_{\mathbb C}\to K_{red}[\Rep GL(m|n)]_{\mathbb C}$ is a homomorphism
  of $\mathfrak{sl}(\infty)$-modules.
  \end{corollary}
  \begin{lemma}\label{tilting}
    \begin{enumerate}
    \item $ds_{m|n}([T(\lambda)])\neq 0$ if and only if $\lambda\in C(m|n)$.
    \item If $ds_{m|n}([T(\lambda)])\neq 0$ and $ds_{m-1|n-1}([T(\lambda)])=0$, then $DS_{m|n}T(\lambda)$ is projective in $\Rep GL(m|n)$.
       \item The set $\{ds_{m|n}([T(\lambda)])\mid\lambda\in C(m|n)\}$ is linearly independent in $ K_{red}[\Rep GL(m|n)]_{\mathbb C}$.
      \end{enumerate}
    \end{lemma}
    \begin{proof} By Lemma \ref{functor} we have $DS_{m|n}T(\lambda)=F_{m|n}T(\lambda)$. Therefore (1) follows from Theorem \ref{CW} (1).
      
      To prove (2) let $P=DS_{m|n}T(\lambda)$. Then we have $DS_{m-1|n-1}T(\lambda)=DS_xP$ for any odd self-commuting $x\in\mathfrak{gl}(m|n)$ of rank $1$, see
      Remark~\ref{unique?}.
      Since the set $X_P=\{y\mid DS_yP\neq 0\}$ is Zariski closed $GL(m)\times GL(n)$-stable subset we obtain $X_P=\{0\}$ and therefore $P$ is projective, see \cite{DS}.

      Now let us prove (3) by induction on $m$. Consider a linear combination $$\sum_{\lambda\in C(m|n)} c_\lambda ds_{m|n}([T(\lambda)])=0.$$ It can be written as
      $$\sum_{\lambda\in C(m-1|n-1)} c_\lambda ds_{m|n}([T(\lambda)])+\sum_{\lambda\notin C(m-1|n-1)} c_\lambda ds_{m|n}([T(\lambda)])=0.$$
      Applying $ds_x$ we get
      $$\sum_{\lambda\in C(m-1|n-1)} c_\lambda ds_{m-1|n-1}([T(\lambda)])=0.$$
      By induction assumption we obtain $c_\lambda=0$ for all $\lambda\in C(m-1|n-1)$. On the other hand, $ds_{m|n}([T(\lambda)])$  for all
      $\lambda\in C(m|n)\setminus C(m-1|n-1)$ is the set of isomorphism classes of all indecomposable projective modules. Hence this set is linearly independent and
      all $c_\lambda=0$.
    \end{proof}
    \begin{corollary}\label{ker} The quotient $\Ker ds_{m-1|n-1}/\Ker ds_{m|n}$ is isomorphic to $S_{m|n}$ as an $\mathfrak{sl}(\infty)$-module.
    \end{corollary}
    \begin{proof} Let us write $ds_{m-1|n-1}=ds_xds_{m|n}$. Then $\Ker ds_{m-1|n-1}/\Ker ds_{m|n}$ is isomorphic to $\Im ds_{m|n}\cap\Ker ds_x$.
      Furthermore  Lemma \ref{tilting} implies that $\Im ds_{m|n}$ is spanned by $ds_{m|n}([T(\lambda])$ for all $\lambda\in C(m|n)$ and $\Im ds_{m|n}\cap\Ker ds_x$
      is spanned by classes of all indecomposable projective modules in $\Rep GL(m|n)$. Therefore the statement follows from Proposition \ref{HPS}.
      \end{proof}

      \begin{lemma}\label{complete} $$\bigcap_{m-n=t}\Ker ds_{m|n}=0.$$ 
      \end{lemma}
      \begin{proof} Suppose $ds_{m|n}([X])=0$ for all $m,n$ such that $m-n=t$. There exists $k$ such that $[X]\in K[\mathcal V_t^k]_{\mathbb C}$.
        But $ds_{m|n}:K[\mathcal V_t^k]_{\mathbb C}\to K[\Rep^k GL(m|n)]_{\mathbb C}$ is injective for
        sufficiently large $m,n$. Therefore $[X]=0$.
      \end{proof}
      Corollary \ref{ker} and Lemma \ref{complete} prove Theorem \ref{main}(1). Indeed, it suffices to put
      $$\mathfrak R^k:=\begin{cases}\ker ds_{k+t-1,k-1}\,\,\text{if}\,\,t\geq 0,\\ \ker ds_{k-1,k-1-t}\,\,\text{if}\,\,t<0.\end{cases}.$$

      Now let us prove Theorem \ref{main}(2). We consider the case $t\geq 0$, the case of negative $t$ is similar.
      Note that $\mathfrak R$ satisfies the following property: for any $u\in \mathfrak R$,
      $e_au=f_au=0$ for all but finitely many $a$. Let $\mathfrak {l}^{-}_s$ (resp, $\mathfrak {l}^{+}_s$) be the Lie subalgebra of $\mathfrak{sl}(\infty)$ generated
      by $e_a,f_a$ for $a< s$ (resp., $a>s)$. Let $M^+_s:=M^{\mathfrak {l}^{-}_s}$. Then $M^+_s$ is a $\mathfrak {l}^+_s$-module. If $M$ is a submodule of $\mathfrak R$
      then $M=\bigcup _{s<0}M_s^+$ by the above property. In particular, if $M,N$ are two submodules of $\mathfrak R$ such that
      $M^+_s=N^+_s$ for all $s<s_0$, then $M=N$. A simple computation shows that for any $s< 0$
      $$\mathfrak R^+_s\simeq\Lambda^{-s-1}((\mathbb V^\vee)_s^+)\otimes \Lambda^{t-s-1}(\mathbb V_s^+).$$
      Note that $\mathfrak{l}^+_s$ is isomorphic to $\mathfrak{sl}(\infty)$ and $(\mathbb V^\vee)_s^+$ and $\mathbb V_s^+$ are isomorphic to the standard and costandard
      $\mathfrak{l}^+_s$-modules respectively. A description of the lattice of all submodules of $\mathfrak R^+_s$ follows immediately from the socle
      filtration of $\mathfrak R^+_s$, see \cite{PS}. Since every layer of this socle filtration is simple, the only
      submodules of $\mathfrak R^+_s$ are members of the socle filtration $\soc ^{r+1}(\mathfrak R^+_s)$ for some $0\leq r\leq -1-s$. Furthermore,
      $\soc ^{r+1}(\Lambda^{-s-1}((\mathbb V^\vee)_s^+)\otimes \Lambda^{t-s-1}(\mathbb V_s^+)$  is cyclic and is generated by a monomial vector
      $x$ such that $c^{r+1}(x)=0,c^r(x)\neq 0$ for the contraction map
      $$c:\Lambda^{k}((\mathbb V^\vee)_s^+)\otimes \Lambda^{t+k}(\mathbb V_s^+)\to\Lambda^{k-1}((\mathbb V^\vee)_s^+)\otimes \Lambda^{t+k-1}(\mathbb V_s^+).$$  
        For any $p\geq 0$ set
        $$v(p):=(w_{t-1}\wedge w_{t-2}\wedge\dots)\otimes (u_{t+p}\wedge u_{t+p-1}\wedge\dots\wedge u_{t+1}\wedge u_{-p-1}\wedge u_{-p-2}\wedge\dots).$$
        By above $\soc ^{r+1}(\mathfrak R^+_s)$  is generated by $v(-r-s-1)$.  Passing to the direct limit for
        $s\to -\infty$ we obtain that every submodule of $\mathfrak R^+_s$ is generated by $v(p)$ for some $p\geq 0$. Thus, we obtain that every submodule
        of $\mathfrak R$ is generated $v(p)$. On the other hand, it is not difficult to see that $\mathfrak R^r$ is generated by $v(r)$. The statement follows.
        \begin{remark}\label{directlimit} The last argument uses presentation of $\mathfrak R$ as a direct limit. Indeed, for the directed system of algebras
          $\dots\subset \mathfrak {l}^{+}_s\subset \mathfrak {l}^{+}_{s-1}\subset \dots$ (here $s\to-\infty$) we get
          $$\mathfrak R=\lim_{\to} \Lambda^{-s+t-1}((\mathbb V^\vee)_s^+)\otimes \Lambda^{-s-1}(\mathbb V_s^+)$$
          for $t\geq 0$ and similarly
          $$\mathfrak R=\lim_{\to} \Lambda^{-s-1}((\mathbb V^\vee)_s^+)\otimes \Lambda^{-s-t-1}(\mathbb V_s^+)$$
          for $t\leq 0$.
          \end{remark}
        \section{Blocks in $\mathcal V_t$ and dimensions of tilting and standard objects.}
        The module $\mathfrak R$ is a weight $\mathfrak{sl}(\infty)$-module. To simplify bookkeeping we
        embed $\mathfrak{sl}(\infty)\hookrightarrow\mathfrak{gl}(\infty)$ and define a $\mathfrak{gl}(\infty)$-action on $\mathfrak R$ in the natural way.
        We fix  the Cartan subalgebra $\mathfrak h$ of the diagonal matrices in $\mathfrak{gl}(\infty)$, choose the basis $\{E_{i,i}\mid \in\mathbb Z\}$ and denote
        by $\{\theta_i \mid i\in\mathbb Z\}$ the dual system in $\mathfrak h^*$. It is easy to compute the weight $\wt(v_\lambda)$ of the monomial vector
        $v_\lambda$. Precisely for a bipartition $\lambda$ define the sets
        $$A(\lambda):=\{\lambda^{\bullet}_i\mid \lambda^{\circ}_i+t-i\neq \lambda^\bullet_j-j\,\forall j\},$$
        $$B(\lambda):=\{\lambda^{\circ}_j\mid \lambda^{\circ}_j+t-j\neq \lambda^\bullet_i-i\,\forall i\}.$$
        It follows immediately from definition that $A(\lambda)$ and $B(\lambda)$ are finite subsets of $\mathbb Z$ and $|B(\lambda)|-|A(\lambda)|=t$.
\begin{example} If
  $\lambda=(\emptyset,\emptyset)$ then
  $A(\lambda)=\emptyset,\ B(\lambda)=\{0,1,\dots,t-1\}$ for $t>0$ and $A(\lambda)=\{-1,\dots,t\},\ B(\lambda)=\emptyset$ for $t< 0$. For $t=0$
$A(\lambda)=B(\lambda)=\emptyset$.
\end{example}
Then we have
\begin{equation}\label{weight}
  \wt(v_\lambda)=-\sum_{a\in A(\lambda)}\theta_a+\sum_{b\in B(\lambda)}\theta_b.
\end{equation}
\begin{theorem}\label{blocks} For a weight $\theta$ of $\mathfrak R$ let $\mathcal V_t^\theta$ denote the full subcategory of $\mathcal V_t$ consisting of objects
  with simple constituents isomorphic to $L(\lambda)$ with $\wt(v_\lambda)=\theta$. Then $\mathcal V_t$ is the direct sum of $\mathcal V_t^\theta$. Moreover,
  $\mathcal V_t^\theta $ is a block in $\mathcal V_t$ for every $\theta$. 
\end{theorem}
\begin{proof} Since $\mathcal V^k_t$ is a highest weight category for every $k$ we have
  $$\Ext^1(L(\lambda),L(\mu))\neq 0 \Rightarrow [V(\lambda):L(\mu)]\neq 0\ \text{or}\ [V(\mu):L(\lambda)]\neq 0.$$
  On the other hand, since $V(\lambda)$ is indecomposable all  its simple constituents lie in the same block of $\mathcal V_t$.
    Combinatorial description of the multiplicities  $[V(\lambda):L(\mu)]\neq 0$ is given in \cite{E}. It is clear from this description that
  $[V(\lambda):L(\mu)]\neq 0$ implies $\wt(v_\lambda)=\wt (v_\mu)$. Let $\sim$ be the equivalence closure of 
  $[V(\lambda):L(\mu)]\neq 0$. Then a simple combinatorial argument implies that $\lambda\sim\mu$ if and only if $\wt(v_\lambda)=\wt(v_\mu)$. 
  \end{proof}
  Let us denote by $\dim M$ the categorical dimension of an object $M$ in $\mathcal V_t$. Since $DS_{m|n}$ is a symmetric monoidal functor it preserves
  categorical dimension. Therefore for every $m,n$ such that $m-n=t$ we have
  \begin{equation}\label{dimension}
    \dim M=\sdim DS_{m|n}M. 
    \end{equation}
    We call weight $\theta$ positive (resp., negative) if $\theta=\sum_{c\in C}\theta_c$, (resp., $\theta=-\sum_{c\in C}\theta_c$). In this definition
    $\theta=0$ is both positive and negative.
    \begin{lemma}\label{dim} \begin{enumerate}
      \item     If $\theta$ is neither positive nor negative, then $\dim M=0$ for every object $M$ in $\mathcal V_t^\theta$.
      \item If $t<0$ and $\theta=\sum_{c\in C}\theta_c$ is positive (resp., $t\geq 0$ and $\theta=-\sum_{c\in C}\theta_c$ is negative), then for every object
        $M$  in $\mathcal V_t^\theta$
        we have
        $\dim M=\kappa(M)q(\theta)$
        for some integer $\kappa(M)$ and
        $$q(\theta)=\frac {\prod_{a<b,a,b\in C}(b-a)}{\prod_{j=1}^{|t|-1}j!}.$$
        \end{enumerate}
      \end{lemma}
      \begin{remark} If $t=0$ the only positive (and negative) weight $\theta$ is zero and $q(\theta)=1$. 
        \end{remark}
  \begin{proof} Say $t\geq 0$. All weights of $\Lambda_{t|0}$ are negative. Since $ds_{t|0}:\mathfrak R\to\Lambda_{t|0}$ is a
    homomorphism of $\mathfrak{sl}(\infty)$-modules $ds_{t|0}[M]=0$ for every $M\in \mathcal V_t^\theta$. Hence the statement is a consequence of
    (\ref{dimension}).
    Similarly for $t<0$ we have  $ds_{0|-t}:\mathcal R\to\Lambda_{0|-t}$ is zero since all weights of $\Lambda_{0|-t}$ are positive. The proof of (1) is complete.

    Let us prove (2). Note in $\Lambda_{t|0}$ and $\Lambda_{0|-t}$ all weight spaces are one-dimensional and the corresponding categories of $GL(|t|)$-supermodules
    are semisimple. Therefore $DS_{t|0}M$ (resp., $DS_{0|-t}M$) is a direct sum of several copies of a certain irreducible representation $W(\theta) $
    of $GL(|t|)$.
    The highest weight
    $\nu(\theta)$ of $W(\theta)$ can be easily expressed in terms of $C=\{c_1>c_2>\dots>c_{|t|}\}$. For $t\geq 0$
    $\nu(\theta)=(c_1+1-t,c_2+2-t,\dots,c_t)$ and for $t<0$ $\nu(\theta)=(c_1+1,\dots,c_{-t}-t)$. Then
    by the Weyl dimension formula we have
    $\sdim W(\theta) =\pm q(\theta)$. This implies (b).
  \end{proof}
  \begin{remark}\label{functorblocks} It is proven in \cite{DS} that $DS_x:\operatorname{Rep}GL(m|n)\to \operatorname{Rep}GL(m-k|n-k)$ maps a block to a block
    corresponding to the same weight of $\mathfrak{gl}(\infty)$.
    Hence  $DS_{m|n}$ induces a functor from  a block $\mathcal V_t^{\theta}$ to the corresponding block $\operatorname{Rep}^{\theta}GL(m|n)$. In particular,
    $DS_{t|0}$ (resp., $DS_{0,|t|}$) annihilates any object in $\mathcal V_t^\theta$ if $\theta$ is not negative (resp., not positive).
    \end{remark}
  \begin{lemma}\label{uniquness} Let $t\geq 0$ (resp., $t<0$).  Then $$\Hom_{\mathfrak{sl}(\infty)}(\mathfrak R,\Lambda^t(\mathbb V^\vee)=\mathbb C,\,\,\text
    {respectively,}\,\,
   \Hom_{\mathfrak{sl}(\infty)}(\mathfrak R,\Lambda^{-t}(\mathbb V))=\mathbb C.$$ 
 \end{lemma}
 \begin{proof} Immediate consequence of Theorem \ref{main}.
 \end{proof}

 Next we are going to construct a homomorphism  $\varphi:\mathfrak R\to \Lambda^t(\mathbb V^\vee)$, (resp., $\varphi:\mathfrak R\to \Lambda^{-t}(\mathbb V)$
 by defining it on the monomial basis $v_\lambda=w_{\lambda^\bullet}\otimes u_{\lambda^{\circ}}$. Let $t>0$ and
 $$u_{\lambda^{\circ}}=u_{i_1}\wedge u_{i_2}\wedge\dots, \quad w_{\lambda^{\bullet}}=w_{j_1}\wedge w_{j_2}\wedge\dots.$$
 If $\wt(v_\lambda)=-\theta_{a_1}-\dots-\theta_{a_t}$ is negative we can write
 $$w_{\lambda^\bullet}=(-1)^{s(\lambda)}w_{a_1}\wedge\dots \wedge  w_{a_t}\wedge w_{i_1}\wedge\dots w_{i_2}\wedge,$$
 and then set
 $$\varphi(v_\lambda):=(-1)^{s(\lambda)}\prod_{i_k\neq -k}(-1)^{i_k}w_{a_1}\wedge\dots\wedge w_{a_t}.$$
 If $\wt(v_\lambda)$ is not negative we set $\varphi(v_\lambda):=0$.
 The easiest way to see that $\varphi$ commutes with action of $\mathfrak{sl}(\infty)$ is to realize it as the direct limit as in Remark \ref{directlimit}.
 Then $\varphi$ is the direct limit of contraction maps $\Lambda^{-s+t}(\mathbb V^\vee)\otimes \Lambda^{-s}(\mathbb V)\to \Lambda^{t}(\mathbb V^\vee)$.

 Similarly, for negative $t$  with $\wt(v_\lambda)=\theta_{a_1}+\dots+\theta_{a_{-t}}$ we write
 $$u_{\lambda^\circ}=(-1)^{s(\lambda)}u_{a_1}\wedge\dots \wedge  u_{a_{-t}}\wedge u_{j_1}\wedge\dots w_{j_2}\wedge,$$
 and we set
 $\varphi(v_\lambda)=(-1)^{s(\lambda)}\prod_{j_k\neq -k}(-1)^{j_k}u_{a_1}\wedge\dots\wedge u_{a_{-t}}$.
 In both cases if $\theta=\wt(\lambda)$ is positive or negative we can write
 $$\varphi(v_\lambda)=(-1)^{r(\lambda)} [W(\theta)],$$
 for certain $r(\lambda)\in\mathbb Z$.
 
 \begin{proposition}\label{standard} If $t\geq 0$ and $\theta$ is negative then dimension of $V(\lambda)$ in $\mathcal V_t^\theta$
   equals $(-1)^{r(\lambda)}q(\theta)$.

   If $t<0$  and $\theta$ is positive then dimension of $V(\lambda)$ in $\mathcal V_t^\theta$
   equals $(-1)^{r(\lambda)+\frac{t(t-1)}{2}+\sum_{i=1}^t a_i}q(\theta)$.
 \end{proposition}
 \begin{proof} First let us see that $ds_{t|0}$ (resp., $ds_{0|-t}$) equals $\varphi$. Indeed, if $\mathbf 1$ denotes the unit object in $\mathcal V_t$ then
   $DS_{t|0}(\mathbf 1)$ (resp.,  $DS_{0|-t}(\mathbf 1)$)
   is the trivial module. Hence $ds_{t|0}$ (resp., $ds_{0|-t}$) coincides with $\varphi$ on the vacuum vector $v_{\emptyset,\emptyset}$. Then the statement follows from
   Lemma \ref{uniquness}.

   Let $t\geq 0$ then $ds_{t|0}(v_\lambda)=(-1)^{r(\lambda)} [W(\theta)]$ and $\sdim W(\theta)=q(\theta)$ since $W(\theta)$ is even. This implies the lemma
   by (\ref{dimension}).
   
   Let $t<0$ then $ds_{0|-t}(v_\lambda)=(-1)^{r(\lambda)} [W(\theta)]$ and the parity of $W(\theta)$ is equal to the parity of the highest weight
   $\nu(\theta)$. The latter is equal
   to the parity of $\sum_{i=1}^t a_i+\frac{t(t-1)}{2}$. Hence the lemma.
   \end{proof}
   \begin{remark}\label{combinatorics} Let us explain how to compute $r(\lambda)$ in terms of weight diagram $f_\lambda$ (see Section 4.1 in \cite{E}).
     Recall that $f_{\lambda}:\mathbb Z\to \{<,>,\times,\circ\}$ is defined as follows:
     \begin{itemize}
     \item  $f_\lambda(i)=\circ$ if $u_i$ and $w_i$ do not occur in $v_\lambda$;
     \item  $f_\lambda(i)=<$ if $u_i$ occurs in $v_\lambda$ and $w_i$ does  not;
     \item  $f_\lambda(i)=>$ if $w_i$ occurs in $v_\lambda$ and $u_i$ does  not;
     \item  $f_\lambda(i)=\times$ if both $u_i$ and $w_i$ occur in  $v_\lambda$.
     \end{itemize}
     We represent $f_\lambda$ graphically by putting symbol $f_\lambda(i)$ into position $i$ on the number line. By definition $f_\lambda(i)=\circ$ for $i>>0$ and
     $f_\lambda(i)=\times$ for $i<<0$. If $\theta=\wt(\lambda)$ is positive then there are
     no symbols $>$ and if it is negative there are no symbol $<$. Symbols $<,>$ are called the core symbols. The core diagram is obtained from $f_\lambda$
     by replacing all $\times$-s by $\circ$-s. Furthermore, $L(\lambda)$ and $L(\mu)$ are in the same block if and only if the core diagrams of $\lambda$ and $\mu$
     coincide. Then $s(\lambda)$ equals the sum over all core symbols of the number of $\times$ to the right of that symbol. Now
     let $$u(\lambda)=\begin{cases}\sum_{i\geq 0,f_\lambda(i)=\times}i\,\text{for}\,t\geq 0,\\ \sum_{i>-t,f_\lambda(i)=\times}i\,\text{for}\,t<0\end{cases}.$$
   Then $r(\lambda)=u(\lambda)+s(\lambda)$.
 \end{remark}
 \begin{proposition}\label{tiltingdim} Let $\theta$ be negative or positive. There is exactly one up to isomorphism tilting object $T(\lambda)$ in the block $\mathcal V_t^\theta$ such that
   $\dim T(\lambda)\neq 0$. This is a unique tilting object in $\mathcal V_t^\theta$ such that $T(\lambda)\simeq V(\lambda)\simeq L(\lambda)$.
 \end{proposition}
 \begin{proof} We start with proving that $\dim T(\lambda)\neq 0$ implies $T(\lambda)\simeq V(\lambda)$ and deal with the case $t\geq 0$. The other case is similar.
   Every $T(\lambda)$ is a direct summand in  $V_t^{\otimes p}\otimes (V_t^*)^{\otimes q}$, therefore it is an indecomposable summand in
   $F_{a_1}\dots F_{a_q}E_{b_1}\dots E_{b_q}\mathbf 1$. Note that $\mathbf 1=V(\emptyset,\emptyset)$. An easy computation shows that for every $\kappa$
   $e_a(v_\kappa)$ and $f_a(v_\kappa)$ is zero, $v_\mu$ or a sum $v_\mu+v_\nu$. Moreover, the latter case is only possible if
   $\wt(\kappa)$ is not positive. If $T(\lambda)$ is not isomorphic to $V(\lambda)$ then for some $k$
   $$F_{a_k}\dots F_{a_q}E_{b_1}\dots E_{b_q} \mathbf 1\in\mathcal V_t^{\theta}$$ for non-positive $\theta$. Then by Remark \ref{functorblocks} for some $k\geq 1$
   $$DS_{t|0}F_{a_k}\dots F_{a_q}E_{b_1}\dots E_{b_q}\mathbf 1=0$$ and hence  $$DS_{t|0}F_{a_1}\dots F_{a_q}E_{b_1}\dots E_{b_q}  \mathbf 1=0.$$ But then
   $DS_{t|0}(T_\lambda)=0$ which implies $\dim T(\lambda)=0$.

   From combinatorial description of $K(\lambda,\mu)$ given in \cite{E} we see that if in $f_\lambda$ there is $\circ$ to the left of some $\times$
   then $K(\lambda,\mu)=1$ for at least one $\mu\neq \lambda$. If the core diagram is fixed then the re is exactly one diagram such that all $\times$-s lie to
   the left of all $\circ$-s. That implies uniqueness of $\lambda$ in every block. We can also characterize $\lambda$ as the minimal weight in the block.
   \end{proof}

\end{document}